\documentclass[12pt,a4paper,british]{scrartcl}
\usepackage[latin1]{inputenc}
\usepackage[british]{babel}
\usepackage{amsmath,amssymb,leftidx}
\usepackage[amsmath,thref,thmmarks]{ntheorem}

\theoremstyle{plain}
\theoremheaderfont{\normalfont\bf}
\theorembodyfont{\itshape}
\theoremseparator{.}
\theoremnumbering{arabic}
\theoremsymbol{}
\newtheorem{thm}{Theorem}
\newtheorem{lem}[thm]{Lemma}

\newtheorem{prop}[thm]{Proposition}

\theoremsymbol{{\Large\ensuremath{_\dashv}}}
\newtheorem{thmprf}[thm]{Theorem}

\theoremstyle{plain}
\theoremheaderfont{\normalfont\bf}
\theorembodyfont{\upshape}
\theoremseparator{.}
\theoremsymbol{}

\newtheorem{example}[thm]{Example}

\theoremstyle{break}
\theoremheaderfont{\normalfont\bf}
\theorembodyfont{\upshape}
\theoremseparator{:}
\theoremsymbol{}
\newtheorem{defi}[thm]{Definition}

\theoremstyle{nonumberplain}
\theoremseparator{:}
\theoremheaderfont{\bf}
\theorembodyfont{\upshape}
\theoremsymbol{{\Large\ensuremath{_\dashv}}}
\newtheorem{proof}{Proof}

\newcommand{\op}[1]{\operatorname{#1}}
\newcommand{\N}{\mathbb{N}}
\renewcommand{\phi}{\varphi}
\renewcommand{\theta}{\vartheta}
\newcommand{\<}{\langle}
\renewcommand{\>}{\rangle}
\newcommand{\al}{\alpha}
\newcommand{\be}{\beta}
\newcommand{\DD}[1]{\mathbb{D}_{2^{#1}}}
\newcommand{\QFSL}{\mathit{QFSL}}
\renewcommand{\|}{\,|\,}
\newcommand{\StarN}{{}^{*}\mathbb{N}}
\newcommand{\Star}[1]{{}^{*}{#1}}
\newcommand{\restr}{\upharpoonright}
\newcommand{\std}{{}^{\circ}}
\newcommand{\C}{\mathcal{C}}

\newcommand{\Bcal}{\mathcal{B}}

\renewcommand{\epsilon}{\varepsilon}
\newcommand{\SL}{\mathit{SL}}

\newcommand{\overp}{\overline{p}}
\newcommand{\overtau}{\overline{\tau}}

\newcommand{\vpt}{v^{\overp,\overtau}}
\newcommand{\SDL}{\op{SDL}}
\newcommand{\jpt}{j^{\overp,\overtau}}
\newcommand{\PSpec}{\op{PSpec}}

\newcommand{\Pcal}{\mathcal{P}}

\title{Representation Theorems for Strong Predicate Exchangeability in Pure Inductive Logic}
\author{Malte S. Klie{\ss}\\
\small{malte@kappaplus.de}}
\date{\today}

\begin{document}
\maketitle
\begin{abstract}
	\noindent In Pure Inductive Logic, the principle of Strong Predicate Exchangeability is a rational principle based on symmetry that sits in between the principles of Predicate Exchangeability and Atom Exchangeability. We will show a de Finetti -- style representation theorem for probability
	functions that satisfy this principle in addition to Unary Language Invariance.
\end{abstract}

Key words: Strong Predicate Exchangeability, Language Invariance, Inductive Logic, Probability Logic, Uncertain Reasoning.

\section{Introduction}

A recurring theme in the study of Pure Inductive Logic
is symmetry.
Amongst such rational principles based on symmetry are the well-known principles of Atom and Predicate Exchangeability\footnote{In the case of Predicate Exchangeability, we would like to point out that while the principle, in various formulations, was accepted as a basic condition for probability functions by both Rudolf Carnap and W.E. Johnson, the principle only recently was studied in more detail, providing representation theorems.}, respectively.

The rational principle of Strong Predicate Exchangeability arose from a generalized version
of the $u^{\overp}$ functions (see \cite[Chapter 29]{ParisVencovskaBook}) that are the building blocks\footnote{By `building
blocks' we mean that any function satisfying some rational principle can be represented as a convex combination of the building block functions.} of functions
satisfying the principles of Atom Exchangeability and Unary Language Invariance. In fact,
these generalized functions turned out to be just too strong to use as building blocks for
Unary Language Invariance.

With this in mind, and especially with the rather artificial motivating example above, one might think that Strong Predicate Exchangeability is a rather
artificial principle and thus might only be of interest as a technical exercise. There is
however an alternative way of obtaining Strong Predicate Exchangeability via a weak
form of Johnson's Sufficientness Postulate (JSP): just as JSP gives rise to Carnap's Continuum of Inductive Methods whose members satisfy Atom Exchangeability, so can we obtain a
collection of functions arising from a weaker form of JSP that are characterized by Strong
Predicate Exchangeability (see \cite{KliessThesis}[Chapter 3.3]). We would then argue
that while there may be no compelling argument for a rational agent to adopt Strong Predicate Exchangeability as a principle in its own right, it would be a consequence for any agent
picking this weaker form of JSP.

The aim of this paper is to present de Finetti -- style representation theorems for
probability functions satisfying Strong Predicate Exchangeability in conjunction with Unary
Language Invariance. We will show that any probabiliy function satisfying Strong Predicate
Exchangeability and Unary Language Invariance can be represented as a convex combination of
functions that in addition also satisfy Weak Irrelevance.

\section{Notation and Principles}

The context of this paper is unary Pure Inductive Logic. We will be concerned
with first order languages that only contain finitely many predicate symbols $P_i$
and countably many constant symbols $a_1,a_2,\dotsc,a_n,\dotsc$ which we can think
of as exhausting the universe. For $k\in\N^{+} = \{n\in\N\| n > 0\}$, let $L_k$ be the
language containing only the predicates $P_1,P_2,\dotsc,P_k$. Unless otherwise stated
we will take our default language $L$ to be the language $L_q$. Let $\QFSL$ denote the
set of quantifier-free sentences of $L$ and $\SL$ the set of sentences of $L$.

	Let $L=L_q$. An \emph{atom} of $L$ is a formula
	$$\al(x) = \bigwedge_{i=1}^q \pm P_i(x),$$
	where $+ P_i(x)$, $- P_i(x)$ stand for $P_i(x)$, $\neg P_i(x)$, respectively.
	Note that for $L=L_q$, there are $2^q$ atoms, denoted $\al_1,\al_2,\dotsc,\al_{2^q}$,
	which we will assume to be in the usual lexicographic ordering.\footnote{In the usual lexicographic ordering, $\al_1 = \bigwedge_{i=1}^q P_i(x)$, $\al_2 = \bigwedge_{i=1}^{q-1} P_i(x)\wedge \neg P_q(x)$, $\al_3 = \bigwedge_{i=1}^{q-2}P_i(x)\wedge\neg P_{q-1}(x)\wedge P_q(x)$, etc., up to $\al_{2^q}=\bigwedge_{i=1}^q\neg P_i(x)$.}
	
	A \emph{state description} of $L$ for $a_1,\dotsc,a_n$ is a sentence
	$$\Theta(a_1,\dotsc,a_n) = \bigwedge_{i=1}^n \al_{h_i}(a_i),$$
	where $\al_{h_i}$ is an $L$-atom for each $h_i$, i.e. $h_i\in\{1,\dotsc,2^q\}$ for each $1\leq i\leq n$.
	
	The following definition will turn out to be very useful in keeping track of the number of negations occurring in an atom.

\begin{defi}
	Let $L=L_q$. Define the function $\gamma_q$ on the atoms of $L$ by
	\begin{gather*}
		\gamma_q(\al) = k \Leftrightarrow \al\text{ contains $k$ negated predicates.}
	\end{gather*}
	We shall drop the index $q$ whenever it is understood from the context.
\end{defi}

\begin{defi}
	Let $L=L_q$ and let $w:\SL\rightarrow[0,1]$. Then $w$ is a probability function
	if $w$ satisfies the following properties.

	\begin{itemize}
		\item[(i)] $w(\top) = 1$,
		
		\item[(ii)] if $\phi\models\neg\theta$, then $w(\phi\vee\theta)= w(\phi)\cdot w(\theta)$,
		
		\item[(iii)] $w\left(\exists x\,\phi(x)\right) = \lim_{n\rightarrow\infty}w\left(\bigvee_{i=1}^n \phi(a_i)\right)$.
		
	\end{itemize}
\end{defi}

Recall that by Gaifman's Theorem (see \cite{Gaifman}) any probability function $w:\QFSL\rightarrow[0,1]$ satisfying
(i) and (ii) above can uniquely be extended to a function $w$ on $\SL$
satisfying (i)-(iii). By the Disjunct Normal Form Theorem it will be enough to define a function
$w$ on state descriptions of $L$.

In the following definitions for rational principles we will assume that $L=L_q$ and
$w$ is a probability function on $\SL$.

{\bf Constant Exchangeability (Ex)}\\
	$w$ satisfies Constant Exchangeability if for all $L$-sentences $\phi$ and all permutations $\sigma$ of $\N^{+}$,
	$$w(\phi(a_1,\dotsc,a_n)) = w(\phi(a_{\sigma(1)},\dotsc,a_{\sigma(n)})).$$

We will assume that all probability functions satisfy Constant Exchangeability.

{\bf Predicate Exchangeability (Px)}\\
	$w$ satisfies
	Px if for all $L$-sentences $\phi$ and all permutations $\sigma$ of the (indices of)
	predicates of $L$,
	$$ w(\phi(P_1,\dotsc,P_m,a_1,\dotsc,a_n)) = w(\phi(P_{\sigma(1)},\dotsc,P_{\sigma(m)},a_1,\dotsc,a_n)).$$

{\bf Atom Exchangeability (Ax)}\\
	$w$ satisfies Atom Exchangeability if for all permutations $\sigma$ of the (indices of)
	atoms of $L$,
	$$ w\left(\bigwedge_{i=1}^n\al_{h_i}(a_i)\right) = w\left(\bigwedge_{i=1}^n\al_{\sigma(h_i)}(a_i)\right).$$

{\bf Weak Irrelevance Principle (WIP)}\\
	$w$ satisfies Weak Irrelevance if whenever $\theta, \phi\in\QFSL$ have no constants
	or predicates in common then
	$$w(\theta\wedge\phi) = w(\theta)\cdot w(\phi).$$

{\bf Unary Language Invariance (ULi)}\\
	Let $w$ be a probability function on $L$.
	Then $w$ satisfies Unary Language Invariance if there is a family of probability functions
	$w^{\mathcal{L}}$, one for each finite unary language $\mathcal{L}$ and satisfying Px + Ex
	such that $w^L = w$ and whenever $\mathcal{L}\subseteq\mathcal{L}'$, then
	$w^{\mathcal{L}'}\restr S\mathcal{L} = w^{\mathcal{L}}$.
	
	We say that $w$ satisfies ULi \emph{with $\Pcal$} for a rational principle $\Pcal$, if each
	$w^{\mathcal{L}}$ in the ULi family also satisfy $\Pcal$.

We will now proceed to define the $P$-spectrum and in turn the principle of Spectrum
Exchangeability, which is the main focus of this paper.

\begin{defi}[$P$-spectrum]
	Let $L=L_q$ and $\Theta(a_1,\dotsc,a_n) = \bigwedge_{i=1}^n \al_{h_i}(a_i)$ be a state description of $L$.
	We can view $\Theta$ as
	\begin{gather*}
		\Theta(a_1,\dotsc,a_n) = \bigwedge_{i=0}^q \bigwedge_{j=1}^{s_i} \al_{h_{ij}}(b_{ij}),
	\end{gather*}
	where each $\al_{h_{ij}}$ for $j\in\{1,\dotsc,s_i\}$ is an atom of $L$ with $\gamma_q(h_{ij}) = i$,
	$n = \sum_{i=0}^q s_i$ and $b_{ij} = a_k$ for that $k$ with  $\Theta(a_1,\dotsc,a_n)\models\al_{h_{ij}}(a_k)$. For each $i\in\{0,\dotsc,q\}$ let $E_i$ be the equivalence relation on $\{1,\dotsc,s_i\}$ given by
	\begin{gather*}
		j E_i k <=> h_{ij} = h_{ik}.
	\end{gather*}
	For $i\in\{0,\dotsc,q\}$ let $M_i$ be the multiset of sizes of equivalence classes of $E_i$. The \emph{$P$-spectrum} of $\Theta$ is the vector
	\begin{gather*}
		\<M_0,M_1,\dotsc,M_q\>.
	\end{gather*}
\end{defi}

\begin{defi}[\bf Strong Predicate Exchangeability, SPx]
Let $L=L_q$ and $w$ a probability function on $L$. $w$ satisfies Strong Predicate Exchangeability if and only if
$w(\Theta) = w(\Phi)$ whenever $\Theta$ and $\Phi$ have the same $P$-spectrum.
\end{defi}

SPx is a strong version of Px in the sense that it implies Px, but Px does not imply SPx.

\begin{prop}
	Let $w$ be a probability function on $L=L_q$ satisfying SPx. Then $w$ satisfies Px.
\end{prop}

\begin{proof}
	Let $\sigma$ be a permutation of predicates, $\Theta$ a state description of $L$ and
	$\sigma\Theta$ the result of permuting all predicates occuring in $\Theta$ according to $\sigma$.
	It is enough to show that $\Theta$ and $\sigma\Theta$ have the same $P$-spectrum.
	
	Suppose $\Theta = \bigwedge_{i=0}^q\bigwedge_{j=1}^{s_i} \al_{h_{ij}}(a_{ij})$ and for some $j,k\in\{1,\dotsc,n\}$, we have
	that $h_{ij}=h_{ik}$, so $j E_i k$ for some $i$. Since $\al_{h_{ij}} = \al_{h_{ik}}$ implies $\sigma\al_{h_{ij}} = \sigma\al_{h_{ik}}$,
	$\sigma$ preserves the $E_i$-equivalence of $j$ and $k$. A similar argument gives that if $j,k\in\{1,\dotsc,s_i\}$ are not $E_i$-equivalent, then
	$\sigma\al_{h_{ij}}\neq \sigma\al_{h_{ik}}$ must hold. Thus we must have that the $P$-spectrum of $\Theta$ is the same
	as the $P$-spectrum of $\sigma\Theta$. By SPx for $w$, $w(\Theta) = w(\sigma\Theta)$.
\end{proof}

The following counter-example will show that Px does not imply SPx.

\begin{example}[Px does not imply SPx]
	Suppose $\Theta(a_1,\dotsc,a_5)$, $\Phi(a_1,\dotsc,a_5)$\\
	are state descriptions of $L_3$ given by
	\begin{align*}
		\Theta &= (P_1\wedge P_2\wedge \neg P_3) \wedge (\neg P_1 \wedge P_2 \wedge P_3) \wedge (\neg P_1\wedge \neg P_2\wedge P_3)\\
			&\qquad \wedge (\neg P_1\wedge \neg P_2 \wedge P_3) \wedge (P_1\wedge \neg P_2\wedge \neg P_3),\\
		\Phi &= (P_1\wedge \neg P_2 \wedge P_3) \wedge (\neg P_1 \wedge P_2 \wedge P_3) \wedge (\neg P_1\wedge P_2\wedge \neg P_3)\\
			&\qquad \wedge (\neg P_1 \wedge P_2 \wedge \neg P_3) \wedge (\neg P_1 \wedge \neg P_2 \wedge P_3).
		\intertext{With the eight atoms of $L_3$ enumerated in lexicographic order, we obtain}		
		\Theta &= \al_2\al_4\al_5\al_7^2,\\
		\Phi &= \al_3\al_5\al_6^2\al_7.
	\end{align*}
	Note that there are six permutations of the atoms of $L_3$ that are induced by Px, and that $\al_2$, $\al_3$, $\al_5$ each have one negated predicate, while
	$\al_4$, $\al_6$, $\al_7$ each have two negated predicates. Suppose $\vec{b}\in\mathbb{D}_8$ and let $\Sigma$ denote the set
	of all permutations of predicates. Then one can verify easily that $w$ given by
	\begin{gather*}
		w = \frac{1}{6}\sum_{\sigma\in\Sigma} w_{\sigma\vec{b}}
	\end{gather*}
	satisfies Px. Notice that $\Phi$ and $\Theta$ have the same $P$-spectrum. We obtain
	\begin{align*}
		w(\Theta) &= \frac{1}{6}\left(b_2b_4b_5b_7^2 + b_3b_4b_5b_6^2 + b_2b_6b_3b_7^2 + b_5b_6b_3b_4^2 + b_5b_7b_2b_6^2 + b_3b_7b_2b_6^2\right),\\
		w(\Phi) &= \frac{1}{6}\left(b_3b_5b_7b_6^2 + b_2b_5b_6b_7^2 + b_5b_3b_7b_4^2 + b_2b_3b_4b_7^2 + b_3b_2b_6b_4^2 + b_5b_2b_6b_4^2\right).
		\intertext{Now letting}
		\vec{b} &= \left\<\frac{1}{19},\frac{2}{19},\frac{4}{19},\frac{5}{19},\frac{2}{19},\frac{3}{19},\frac{1}{19},\frac{1}{19}\right\>,
	\end{align*}
	we clearly have $\vec{b}\in\mathbb{D}_8$ and we obtain
	\begin{align*}
		w(\Theta) &= \frac{1094}{6\cdot 19^5},\\
		w(\Phi) &= \frac{1224}{6\cdot 19^5},
	\end{align*}
	which clearly gives $w(\Theta)\neq w(\Phi)$ and thus $w$ cannot satisfy SPx.
\end{example}

Similarly we can observe that any function satisfying Ax must also satisfy SPx, as we can easily
see that any permutation $\sigma$ of atoms such that $\gamma_q(\al) = \gamma_q(\sigma\al)$
preserves $P$-spectra. Conversely whenever $\Theta$ and $\Phi$ are state descriptions
with the same $P$-spectrum then there is a bijection between the atoms of $\Theta$ and those
of $\Phi$ that can be extended to a permutation of atoms. It is easy to see that there are functions
satisfying SPx, but not Ax as any permutation that permutes atoms with different numbers of negations
can be used to construct a counter-example in just the same fashion as the above example.

\section{A representation theorem for ULi with SPx}

The first step towards the desired representation theorem is defining the class of functions
$\vpt$ that we will show to be the basic building blocks for functions satisfying ULi with SPx.
It is apparent from the definition that these functions are generalizations of the $u^{\overp}$
functions that form the building blocks for Ax + ULi, see e.g. \cite{ParisVencovskaBook}.
As was discussed in \cite{KliessThesis} the classes of $\vpt$ and $u^{\overp}$ share a number
of properties, of which we will quote (without proof) the most relevant for our purposes.

Let $\overp$, $\overtau$ be countable sequences $\overp = \<p_0,p_1,p_2,\dotsc\>$, 
$\overtau = \<\tau_0,\tau_1,\tau_2,\dotsc\>$ and let $\Bcal$ be the set
\begin{multline*}
	\Bcal = \left\{\<\overp,\overtau\>\left|\vphantom{\sum}\right. p_i\geq 0\text{ for } i\geq 0, p_1\geq p_2\geq p_3\geq\dotsc\geq p_n\geq\dotsc, \sum_{i\geq 0} p_i = 1, \right.\\\left.\vphantom{\sum_{i\in\N}} \tau_j\in[0,1]\text{ for } j\geq 1 \text{ and } \tau_0 \text{ a normalized $\sigma$-additive measure on }[0,1]\right\}.
\end{multline*}

\begin{defi}\label{defvpt}
	Let $L=L_q$ and let $\<\overp,\overtau\>\in\Bcal$ and let $\Theta(\vec{a}) = \Theta(a_1,\dotsc,a_m)$ be a
	state description of $L$ with
	\begin{gather*}
		\Theta(a_1,\dotsc,a_m) = \bigwedge_{i=1}^m \al_{h_i}(a_i).
	\end{gather*}
	Define the functions $\jpt_L$, $\vpt_L$ as follows: Let $\vec{c}$ be a sequence in $\N$. Then
	\begin{align*}
		\jpt_L (&\Theta(\vec{a}),\<c_1,\dotsc,c_m\>)\\
		&= \begin{cases}
			\jpt_L(\Theta^{-},\<c_1,\dotsc,c_{m-1}\>) \cdot p_{c_m} \cdot b_{h_m,c_{m}}&\text{if $c_m = 0$ or $c_l\neq c_m$ for all $l<m$},\\
			{}&{}\\
			\jpt_L(\Theta^{-},\<c_1,\dotsc,c_{m-1}\>) \cdot p_{c_m} & \text{if $c_l=c_m\neq 0$ for some}\\
			&\text{$l<n$ and $h_l = h_m$,}\\
			{}&{}\\
			0&\text{otherwise,}
		\end{cases}
	\end{align*}
	with $\Theta^{-}=\Theta^{-}(a_1,\dotsc,a_{m-1})$ the unique state description such that
	\begin{gather*}
		\Theta(a_1,\dotsc,a_m) \models\Theta^{-}(a_1,\dotsc,a_{m-1})
		\intertext{and}
		b_{h_m,c_m} = \begin{cases}\tau_{c_m}^{\gamma_q(h_m)}(1-\tau_{c_m})^{\gamma_q(h_m)}&\text{if $c_m>0$,}\\
		\int_{[0,1]}x^{\gamma_q(h_m)} (1-x)^{q-\gamma_q(h_m)}\,d\tau_0(x)&\text{if $c_m$ = 0.}\end{cases}
	\end{gather*}
	Define $\vpt_L$ on state descriptions of $L$ by
	\begin{gather}
		\vpt_L(\Theta(\vec{a})) = \sum_{\vec{c}}\jpt_L(\Theta(\vec{a}),\vec{c}).\label{def_vpt}
	\end{gather}
\end{defi}

Note that this definition is slightly different than the original definition  given in
\cite{KliessThesis}: in the original definition we allowed even the $\tau_i$ for $i>0$ to
be normalized $\sigma$-additive measures instead of single point measures.

The following theorem lists some properties of the $\vpt$ functions. The proof is rather lengthy and technical, and in the case of Weak Irrelevance even requires a detour via a different definition
for the $\vpt$. As we believe the techniques required for the proof do not provide any benefit
for the task at hand, and for the sake of brevity, we would like to refer the reader to Chapter
3 of \cite{KliessThesis}, where detailed proofs for the claims are given.

\begin{thmprf}\label{thmvpt}
	Let $\<\overp,\overtau\>\in\Bcal$. Then $\vpt$ is a probability function satisfying 
	the principles Ex, SPx, ULi and WIP.
\end{thmprf}

We can now state the representation theorem for functions satisfying ULi with SPx.
In the proof we will be using methods from non-standard analysis, and in particular Loeb measure
theory, to obtain the desired results. See e.g. \cite{Cutland} for details.

\begin{thm}\label{RepThm}
	Let $w$ be a probability function on $L_q$. Then $w$ satisfies ULi with SPx if and only if there exists a normalised $\sigma$-additive measure $\mu$ on
	$\Bcal$ such that
	\begin{gather}
		w = \int_{\Bcal} \vpt\,d\mu(\<\overp,\overtau\>).\label{eqRepThm1}
	\end{gather}
\end{thm}

\begin{proof}
	By Theorem \ref{thmvpt}, it is straightforward to show that any function of the form \eqref{eqRepThm1} satifies ULi with SPx.

	For the reverse, let $w$ be a probablity function on $L_q$ satisfying ULi with SPx. Working in a non-standard
	universe, let $\nu\in\Star{\N}\setminus\N$. Then by ULi there exists an extension 
	$w^{L_\nu}$ of $w$ to $L_\nu$.
	
	We can write $w^{L_\nu}$ as
	\begin{gather}
		w^{L_\nu}(\Theta(a_1,\ldots,a_n)) = \sum_{\substack{\Phi(a_1,\dotsc,a_\nu)\in\SDL_\nu\\\Phi\models\Theta}}w^{L_\nu}(\Phi(a_1,\dotsc,a_\nu)).\label{eqOne}
	\end{gather}
	
	Since $w^{L_\nu}$ satisfies SPx, we can rephrase \eqref{eqOne}: fix some $P$-spectrum $\hat{\nu}$ of a state description $\Upsilon(a_1,\dotsc,a_\nu)$ of $L_\nu$. Then
	for each state description $\Phi(a_1,\dotsc,a_\nu)$ such that $\PSpec(\Phi) = \PSpec(\Upsilon) = \hat{\nu}$ we have $w^{L_\nu}(\Phi) = w^{L_\nu}(\Upsilon)$.
	
	Letting $\overline{\Upsilon} = \{\Phi(a_1,\dotsc,a_\nu) \| \PSpec(\Phi) = \PSpec(\Upsilon)\}$ for a fixed state description $\Upsilon$, and partitioning the set of all state descriptions
	of $L_\nu$ for constants $a_1,\dotsc,a_\nu$ by a choice of $\overline{\Upsilon}$ we obtain
	\begin{align*}
		w^{L_\nu}(\Theta(a_1,\dotsc,a_n)) &=
		\sum_{\overline{\Upsilon}}\sum_{\substack{\Phi\in\overline{\Upsilon}\\\Phi\models\Theta}} w^{L_\nu}(\Phi(a_1,\dotsc,a_\nu))\\
		&=\sum_{\overline{\Upsilon}}|\{\Phi\in\overline{\Upsilon}\|\Phi\models\Theta\}|\cdot w^{L_\nu}(\Upsilon(a_1,\dotsc,a_\nu))\\
		&=\sum_{\overline{\Upsilon}} \frac{|\{\Phi\in\overline{\Upsilon}\|\Phi\models\Theta\}|}{|\overline{\Upsilon}|}\cdot |\overline{\Upsilon}|\cdot w^{L_\nu}\left(\Upsilon\right)\\
		&= \sum_{\overline{\Upsilon}} \frac{|\{\Phi\in\overline{\Upsilon}\|\Phi\models\Theta\}|}{|\overline{\Upsilon}|}\cdot w^{L_\nu}\left(\bigvee\overline{\Upsilon}\right),
	\end{align*}
	where $\bigvee\overline{\Upsilon}$ is the disjunction over all state descriptions in $\overline{\Upsilon}$.
	
	We will show that $\frac{|\{\Phi\in\overline{\Upsilon}\|\Phi\models\Theta\}|}{|\overline{\Upsilon}|}$
	is a probability function of the form $\vpt$ for some $\<\overp, \overtau\>\in\Bcal$.
	
	First note that if $\Phi\in\overline{\Upsilon}$ then there exists a permutation $\sigma$
	of the atoms of $L_\nu$ preserving $P$-spectra and a permutation $\rho$ of the constants
	$a_1,\dotsc,a_\nu$ such that $\Phi(a_1,\dotsc,a_\nu) = \sigma\Upsilon(a_{\rho(1)},\dotsc,a_{\rho(\nu)})$. We can then view $\frac{|\{\Phi\in\overline{\Upsilon}\|\Phi\models\Theta\}|}{|\overline{\Upsilon}|}$
	as the probability of obtaining some $\Phi(a_1,\dotsc,a_\nu)\models\Theta(a_1,\dotsc,a_n)$
	with $\Phi\in\overline{\Upsilon}$ by randomly picking permutations $\sigma$ and $\rho$ such
	that $\Phi(a_1,\dotsc,a_\nu) = \sigma\Upsilon(a_{\rho(1)},\dotsc,a_{\rho(\nu)})$.
	
	Fix some $\Upsilon\in\SDL_\nu$ with $\PSpec(\Upsilon) = \hat{\nu} = \<s_0,s_{1,1},\dotsc,s_{1,i_1},s_{2,1},\dotsc,s_{2,i_2},\dotsc,s_\nu\>$, where
	the $s_{i,j}$ are the sizes of the equivalence classes in the $P$-spectrum of $\Upsilon$.\footnote{Note that $s_0$ and $s_\nu$ only have a single index, as there can be only one equivalence class of constants instantiating an atom with no negations or only negations occurring, respectively.} 
	Then the probability of picking a constant of a particular equivalence class at random is
	$\frac{s_{i,j}}{\nu}$, where $s_{i,j}$ is the size of this equivalence class according
	to $\PSpec(\Upsilon)$. Let $p_1\geq p_2\geq p_3\geq\dotsb\geq p_N$ be the list of these
	probabilities, in decreasing order, and let $s'_1,\dotsc,s'_N$ be an enumeration of the equivalence classes such that $p_i = \frac{s'_i}{\nu}$. Here, $N$ is the number of equivalence classes in $\hat{\nu}$. Note that $N$ may very well be infinite. Then the probability of randomly picking (with replacement) constants $a_1,\dotsc,a_n$ from equivalence classes $s'_{c_1},\dotsc,s'_{c_n}$,
	for $1\leq c_1,\dotsc,c_n\leq N$, is $\prod_{i=1}^n p_{c_i}$.\footnote{Just as in the proof of the representation theorem for ULi (see \cite{KliessParis1}) it will suffice to consider picking constants with replacement, as the difference to picking constants without replacement will disappear once we take standard parts.}
	
	We now need to take care of the specific atom picked for each of the classes $s'_{c_1},\dotsc,s'_{c_n}$. Again, as we are going to take standard parts and thus
	the difference between picking atoms with and without replacement will disappear,
	we can treat picking an atom to represent each class as independent of each other.
	Note that the only restriction for picking an atom is that if the original atom
	in $\Upsilon$ representing this class has $j$ negated predicates occurring, so must
	the atom we pick at random have $j$ predicates in order to preserve the $P$-spectrum
	of $\Upsilon$. The atom we pick does not even have to occur in $\Upsilon$. We can view
	randomly picking an atom in this way just as picking predicates without replacement
	from the original atom, determining the atom we want to replace it with in this way.
	
	Consider the equivalence class $s'_{c_m}$ for some $m\leq n$. Suppose that the atom
	of $L_\nu$ this equivalence class represents has $j$ negated predicates occurring.
	Suppose that $\Theta(a_1,\dotsc,a_n)\models\al_{h_m}(a_m)$ with $\gamma_q(\al_{h_m}) = r$.
	Then the probability that picking an atom $\be$ of $L_\nu$ with $\gamma_\nu(\be) = j$ and
	$\be(x)\models\al_{h_m}(x)$ is determined as follows: picking predicates randomly from $\be$ without replacement and if the predicate $P_i$ for $1\leq i\leq q$
	occurs positively in $\al_{h_m}$, we obtain a factor $\frac{\nu-j-k_i}{\nu-(i-1)}$, where
	$k_i$ is the number of predicates $P_t$, $1\leq t<i$ that occur positively in $\al_{h_m}$.
	Similarly if $P_i$ occurs negatively we obtain a factor $\frac{j-k_i}{\nu-(i-1)}$ with $k_i$ the
	number of predicates occurring negatively in $\al_{h_m}$. Once we have picked $q$ predicates
	this way, we are left with an arbitrary choice of predicates to pick, resulting in a factor
	of $1$ for the choice of $P_t$, $t>q$. By commutativity, we obtain the probability as
	\begin{gather}
		\prod_{i=0}^{r-1}\frac{j-i}{\nu-i} \cdot\prod_{i=0}^{q-r-1}\frac{\nu-j-i}{\nu-i-r}.\label{tau_pre_def}
	\end{gather}
	As we will be taking standard parts and $i$ only takes finite values, we may write
	\eqref{tau_pre_def} as
	\begin{gather}
		\left(\frac{j}{\nu}\right)^{r}\cdot\left(\frac{\nu-j}{\nu}\right)^{q-r} = \tau^r(1-\tau)^{q-r},\label{tau_def}
	\end{gather}
	for some $\tau\in\Star[0,1]$, as the difference between \eqref{tau_pre_def} and \eqref{tau_def} will disappear once we have taken standard parts.
	
	Thus for the probability that a fixed $\Phi\in\overline{\Upsilon}$ is such that
	$\Phi\models\bigwedge_{i=1}^n\al_{h_i}(a_i)$ we obtain a factor of $\prod_{i=1}^n p_{c_i}$, where $s'_{c_i}$
	is the equivalence class of $a_i$ in the $P$-spectrum of $\Phi$. Furthermore we obtain
	a factor of $\tau_{k}^{\gamma_q(\al_{h_i})}(1-\tau_{k})^{q-\gamma_q(\al_{h_i})}$
	for each $k\in\{c_1,\dotsc,c_n\}$. The reason we only have one occurrence of this factor
	is that once we have picked an atom for one constant instantiating a particular equivalence
	class of the $P$-spectrum, each further constant from the same class must instantiate the same
	atom. For the same reason, the probability must be $0$ if we have $c_i=c_j$ for some
	$1\leq i,j\leq n$ but $\al_{h_i}$ and $\al_{h_j}$ are two different atoms. In contrast to
	the Px case we can treat the picking of predicates described above as independent for
	each equivalence class of atoms: we are in fact just picking atoms of $L_\nu$ with
	the correct number of negations occurring while ensuring their restriction to $L_q$ yields
	the desired atom. As the probability of picking the same atom of $L_\nu$ for any two equivalence	classes is infinitesimal, we can waive the difference between picking these
	atoms with and without replacement, resulting in the probability of each atom being picked independent of the choice of equivalence classes. Thus we obtain
	\begin{gather*}
		\prod_{i=1}^n p_{c_i} \cdot \prod_{k\in\{c_1,\dotsc,c_n\}}\tau_{k}^{\gamma_q(\al_{h_i})}(1-\tau_{k})^{q-\gamma_q(\al_{h_i})}
	\end{gather*}
	as the probability, in case that whenever $c_i = c_j$ then $h_i = h_j$. But this is just
	a non-standard version of
	$\jpt(\bigwedge_{i=1}^n\al_{h_i}(a_i),\<c_1,\dotsc,c_n\>)$ for
	$\overp = \<0,p_1,p_2,\dotsc\>$ and $\overtau = \<\tau_0,\tau_1,\tau_2,\dotsc,\>$ for
	an arbitrary measure $\tau_0$.
	
	We now obtain
	\begin{gather*}
		\frac{|\{\Phi\in\overline{\Upsilon}\|\Phi\models\Theta\}|}{|\overline{\Upsilon}|} = 
		\sum_{\substack{\<c_1,\dotsc,c_n\>\\c_i\geq 1}}\Star\jpt(\Theta(a_1,\dotsc,a_n),\<c_1,\dotsc,c_n\>) = \Star\vpt(\Theta(a_1,\dotsc,a_n)),
	\end{gather*}
	for a fixed $\Theta$ such that $\PSpec(\Theta)=\hat{n}$ and the $\Star{}$ indicating the
	non-standard versions of the functions defined in Definition \ref{defvpt}.
	
	Define a measure $\mu$ on the $P$-spectra of state descriptions of $L_\nu$ by
	$\mu(\{\hat{\nu}\}) = w^{L_\nu}\left(\bigvee\overline{\Upsilon}\right)$
	and we obtain
	\begin{align*}
		 \sum_{\hat{\nu}}\frac{|\{\Phi\in\overline{\Upsilon}\|\Phi\models\Theta\}|}{|\overline{\Upsilon}|}\cdot w^{L_\nu}\left(\bigvee\overline{\Upsilon}\right) &=
		 \int_{\hat{\nu}}\frac{|\{\Phi\in\overline{\Upsilon}\|\Phi\models\Theta\}|}{|\overline{\Upsilon}|}\,d\mu(\hat{\nu})\\
		 &= \int_{\hat{\nu}}\Star\vpt(\hat{n})\,d\mu(\hat{\nu}).
	\end{align*}
	
	Now taking standard parts we obtain
	\begin{align}
		w^{L_q}(\Theta(a_1,\dotsc,a_n)) &=
		\leftidx{^{\circ}}{\left(\vphantom{\int_{\hat{\nu}}}\right.}
			\left.\int_{\hat{\nu}}\Star\vpt(\Theta(a_1,\dotsc,a_n))\,d\mu(\hat{\nu})\right)\nonumber\\
		&= \leftidx{^{\circ}}{\left(\vphantom{\int_{\hat{\nu}}}\right.}
					\left.\int_{\<\overp,\overtau\>}\Star\vpt(\Theta(a_1,\dotsc,a_n))\,d\mu(\<\overp,\overtau\>)\right),\label{initial_problem}
	\end{align}
	as $\<\overp,\overtau\>$ represents $\hat{\nu}$.
	
	We can now use Loeb measures to move the operation of taking standard parts inside the
	integral and claim that $\std(\Star\vpt(\Theta(a_1,\dotsc,a_n))) = v^{\overline{p'},\overline{\tau'}}(\Theta(a_1,\dotsc,a_n))$ for some standard versions
	$\overline{p'}$ of $\overp$ and $\overline{\tau'}$ for $\overtau$.
	
	It remains to find $\<\overline{p'},\overline{\tau'}\>\in\Bcal$ such that
	\begin{gather}
		\std(\Star\vpt(\Theta(a_1,\dotsc,a_n))) = v^{\overline{p'},\overline{\tau'}}(\Theta(a_1,\dotsc,a_n))\label{eq_stdvpt}
	\end{gather}
	First notice that $\Star\jpt(\Theta(a_1,\dotsc,a_n),\vec{c})$ is a finite product for each $\vec{c}$ and
	so we can move the standard part operation all the way to the individual factors.
	This provides us with an obvious candidate for $\overline{\tau'}$: take $\tau'_i = \std\tau_i$ for each $i\in\N$, $i\geq 1$.
	
	For the $p_i$ there are three cases to consider.
	If there are only finitely many equivalence classes in $\hat{\nu}$,
	there exists some $n\in\N$ such that $p_i = 0$ for each $i> n$. In this case we have 
	$\sum_{i=1}^n \std p_i = 1$ (as $\sum_{i\in\N}p_i = 1$) and we can
	take $p'_i = \std p_i$ for each $1\leq i\leq n$ and let $p'_i = 0$ otherwise.
	Then clearly we have
	\begin{gather}
		\std(\Star\jpt(\Theta(a_1,\dotsc,a_n),\vec{c})) = j^{\overline{p'},\overline{\tau'}}(\Theta(a_1,\dotsc,a_n),\vec{c})\label{eq_stdjpt}
	\end{gather}
	and \eqref{eq_stdvpt} holds for $\<\overline{p'},\overline{\tau'}\>$ as there are only
	finitely many instances of $j^{\overline{p'},\overline{\tau'}}$ occurring.
	
	Otherwise we have that $p_i>0$ for each $i\in\N$. Suppose that in this case we have $\sum_{i\in\N}\std p_i = 1$, i.e. there is no weight carried by the $p_i$ with non-standard indices $i\in\StarN\setminus\N$. Then we can take $p'_i = \std p_i$ for each $i\in\N$, $i\geq 1$ and $p_0 = 0$ for \eqref{eq_stdjpt} to hold: we can interpret $\vpt(\Theta(a_1,\dotsc,a_n),\vec{c})$ as an instance of integration by a discrete measure: Let $\rho$ be the product measure giving
	weight $\prod_{i=1}^n p_{c_i}$ to the point $\vec{c}\in\StarN^n$ and let $f$ be the function defined by
	\begin{gather*}
		f(\Theta(a_1,\dotsc,a_n),\vec{c}) = \prod_{i\in\{c_1,\dotsc,c_n\}}\tau_i^{\gamma_q(h_i)}(1-\tau_i)^{q-\gamma_q(h_i)}.
	\end{gather*}
	Then we can write
	\begin{gather}
		\Star\vpt(\Theta(a_1,\dotsc,a_n),\vec{c}) = \sum_{\vec{c}\in\StarN^n}\jpt(\Theta(a_1,\dotsc,a_n),\vec{c}) = \int_{\StarN}f(\Theta(a_1,\dotsc,a_n),\vec{c})\,d\rho(\vec{c}).\label{eqLoeb2}
	\end{gather}
	Now taking the standard part of the integral in \eqref{eqLoeb2} and applying Loeb measure theory in this situation, we obtain
	as Loeb measure $\rho^L$ the measure on $\N^n$ giving weight $\prod_{i=1}^n p'_{c_i}$
	to the point $\vec{c}\in\N^n$, observing that whenever $\vec{c}\in\StarN^n$ is such that
	for some $j\leq n$, $c_j\in\StarN\setminus\N$, we have $\leftidx{^{\circ}}{\left(\vphantom{\prod_{i=1}^n}\right.}\left.\prod_{i=1}^n p_{c_i}\right) = 0$ and thus \eqref{eq_stdvpt} holds.
	
	Lastly suppose that $\sum_{i\in\N}\std p_i < 1$. In this case we immediately have that
	while $\Star\vpt(\top) = 1$, simply taking standard parts of each $p_i$ as above
	will lead to $v^{\overline{p'},\overline{\tau'}}(\top) < 1$. The obvious problem now is
	to attribute the weight `lost' by simply taking standard parts to the $p_i'$ with standard
	indices. As $p_0$ has not been used yet, we will put all this weight into $p'_0$ and it
	remains to find a measure $\tau_0$ such that the equation \eqref{eq_stdvpt} holds.

	In more detail, let $p'_i = \std p_i$ for each $i\in\N$, $i\geq 1$ and let $p'_0 = 1 - \sum_{i\in\N} \std p_i$. Fix some $\vec{c}$ such that there exists $j$ with $\std p_{c_j} = 0$ and $p_{c_j}>0$. Let $\C(\vec{c},j)$ be the collection 
	\begin{gather*}
		\C(\vec{c},j) = \{\vec{e} \| \forall i\, (i\neq j \rightarrow e_i = c_i) \wedge (i = j \rightarrow p_{e_i}>0 \wedge \std p_{e_i} = 0)\},
	\end{gather*}
	i.e. the sequences in $\C(\vec{c},j)$ only differ in the $j$th component, and each of the disappearing $p_i$ occurs in one of the sequences.
	
	Then we obtain
	\begin{multline*}
		\leftidx{^{\circ}}{\left(\vphantom{\sum_{\vec{e}\in\C(\vec{c},j)}}\right.}\left.\sum_{\vec{e}\in\C(\vec{c},j)}\prod_{i=1}^n p_{e_i} \cdot \prod_{k\in\{e_1,\dotsc,e_n\}}\tau_{k}^{\gamma_q(\hat{h}_k)}(1-\tau_{k})^{q-\gamma_q(\hat{h}_k)}\right) =\\ \leftidx{^{\circ}}{\left(\vphantom{\prod_{\substack{i=1\\i\neq j}}^n}\right.}\left.\prod_{i=1}^n p_{c_i} \cdot \prod_{k\in\{c_1,\dotsc,c_n\}\setminus\{c_j\}}\tau_{k}^{\gamma_q(\hat{h}_k)}(1-\tau_{k})^{q-\gamma_q(\hat{h}_k)}\right)\cdot \leftidx{^{\circ}}{\left(\vphantom{\sum_{\substack{p_j > 0\\\std p_j = 0}}}\right.}\left.\sum_{\substack{p_j > 0\\\std p_j = 0}} p_j\cdot \tau_j^{\gamma_q(h_j)}(1-\tau_j)^{q - \gamma_q(h_j)}\right),
	\end{multline*}
	where $\hat{h}_k$ here is that $h_s$ such that $k = c_s$ ($k=e_s$, respectively).
	
	Let $\tau_0$ be a measure on $\Star[0,1]$ such that $\tau_0$ gives weight $p_j/p'_0$ to
	the point $\tau_j$ for each $j$ such that $p_j>0$ and $\std p_j = 0$, and weight $0$ to all other points. Then we obtain
	\begin{align}
		\leftidx{^{\circ}}{\left(\vphantom{\sum_{\substack{p_j > 0\\\std p_j = 0}}}\right.}\left.\sum_{\substack{p_j > 0\\\std p_j = 0}} p_j\cdot \tau_j^{\gamma_q(h_j)}(1-\tau_j)^{q - \gamma_q(h_j)}\right) &= p'_0\cdot \leftidx{^{\circ}}{\left(\vphantom{\int_{\Star[0,1]}}\right.}\left.\int_{\Star[0,1]}x^{\gamma_q(h_j)}(1-x)^{q - \gamma_q(h_j)}\,d\tau_0(x)\right)\nonumber\\
		&= p'_0 \int_{[0,1]}x^{\gamma_q(h_j)}(1-x)^{q - \gamma_q(h_j)}\,d\tau^L_0(x),\label{factor_tau_0}
	\end{align}
	where $\tau^L_0$ is the corresponding Loeb measure to $\tau_0$. We can continue in this way
	until all such $p_i$ have been collected into a factor of the form \eqref{factor_tau_0}.
	
	Letting $\tau'_0 = \tau^L_0$ a straightforward calculation shows that
	\begin{gather}
		\std\Star\vpt(\Theta) = v^{\overline{p'},\overline{\tau'}}(\Theta),\label{vpt_solution}
	\end{gather}
	as we can now reduce the left-hand side to the situation given in \eqref{eqLoeb2}, with $p_0 = p_0'$, $\tau_0$ as defined above, and $\sum_{i\in\N}\std p_i = 1$ : as the $p_i > 0$ with $\std p_i = 0$ are now collected into $p_0$, just let these $p_i = 0$.
	
	Going back to \eqref{initial_problem}, we obtain together with \eqref{vpt_solution}
	\begin{align*}
		w^{L_q}(\Theta(a_1,\dotsc,a_n)) &= \leftidx{^{\circ}}{\left(\vphantom{\int_{\hat{\nu}}}\right.}
							\left.\int_{\<\overp,\overtau\>}\Star\vpt(\Theta(a_1,\dotsc,a_n))\,d\mu(\<\overp,\overtau\>)\right)\\
		&= \int_{\<\overline{p'},\overline{\tau'}\>}v^{\overline{p'},\overline{\tau'}}(\Theta(a_1,\dotsc,a_n))\,d\mu^L(\<\overline{p'},\overline{\tau'}\>),
	\end{align*}
	where $\mu^L$ is the Loeb measure corresponding to $\mu$.
\end{proof}

\begin{lem}
	The $\vpt$ are the only functions satisfying ULi with SPx + WIP.
\end{lem}

\begin{proof}
	Let $w$ be a probability function satisfying ULi with SPx + WIP and let $\theta\in\QFSL$.
	Using ULI, extend $w$ to some $w'$ on a language $L'$ large enough so that we can permute
	the predicates and constants of $\theta$ to obtain $\theta'\in\QFSL'$ with no predicates or
	constants in common with $\theta$. As SPx implies Px, and by ULi we obtain
	$w(\theta) = w'(\theta) = w'(\theta')$. Then we obtain
	\begin{align*}
		0 &= 2(w'(\theta\wedge\theta') - w'(\theta)\cdot w'(\theta'))\\
		&= \int_{\Bcal}v^{\overline{r},\overline{\rho}}_{L'}(\theta\wedge\theta')\,d\mu(\<\overline{r},\overline{\rho}\>)
		- 2\int_{\Bcal}v^{\overline{r},\overline{\rho}}_{L'}(\theta)\,d\mu(\<\overline{r},\overline{\rho}\>)\cdot \int_{\Bcal}\vpt_{L'}(\theta')\,d\mu(\<\overp,\overtau\>)\\
		&\qquad +\int_{\Bcal}\vpt_{L'}(\theta\wedge\theta')\,d\mu(\<\overp,\overtau\>)\\
		&= \int_{\Bcal}\int_{\Bcal}\left(v^{\overline{r},\overline{\rho}}(\theta)^2 - 2v^{\overline{r},\overline{\rho}}(\theta)\cdot\vpt(\theta) + \vpt(\theta)^2\right)\,d\mu(\<\overline{r},\overline{\rho}\>)\,d\mu(\<\overp,\overtau\>)\\
		&= \int_{\Bcal}\int_{\Bcal}\left(v^{\overline{r},\overline{\rho}}(\theta) - \vpt(\theta)\right)^2\,d\mu(\<\overline{r},\overline{\rho}\>)\,d\mu(\<\overp,\overtau\>),
	\end{align*}
	by the Representation Theorem \ref{RepThm}. As the function under the integral is non-negative we must have a measure $1$ subset of $\Bcal$ such that $\vpt_{L'}$ is constant on this set
	for each $\theta\in\QFSL$ and therefore we must have $w' = \vpt_{L'}$ for any $\<\overp,\overtau\>$ in this set. By ULi we obtain that $w = w'\restr\SL = \vpt_L$, as required.
\end{proof}

\section{The General Representation Theorem}

Just as for the probability functions satisfying Predicate and Atom Exhcangeability\footnote{See e.g. \cite{KliessParis1} for the result concerning Predicate Exchangeability and \cite[Chapter 34]{ParisVencovskaBook} for the result concerning Atom Exchangeability.}, respectively,
we obtain a similar result for Strong Predicate Exchangeability:

\begin{thm}
	Let $w$ be a probability function on $L$. Then $w$ satisfies SPx if and only if
	there are $\lambda\geq 0$ and probability functions $w_1, w_2$ satisfying ULi with SPx
	such that
	\begin{gather*}
		w = (1+\lambda)w_1 - \lambda w_2.
	\end{gather*}
\end{thm}

In the proof, we will use a slightly different formulation\footnote{In fact, the functions $\vpt_{n,L_q}$ are instrumental to the proof of Theorem \ref{thmvpt}.} for the $\vpt$ functions:
Let $\vpt_{n,L_q}$ be defined by
\begin{gather}\label{repvptn}
	\vpt_n = \sum_{e\in Z^q_n}\prod_{r=1}^n \tau_r^{\gamma_q(e(r))}(1-\tau_r)^{q-\gamma_q(e(r))}\cdot w_{e(\overp)},
\end{gather}
where $Z^q_n$ is the set of all functions $e:\{1,\dotsc,n\}\rightarrow\{1,\dotsc,2^q\}$ and
$e(\overp)\in\DD{q}$ is given by
\begin{gather*}
	e(\overp) = \<f^q(1,0)R_{\overp,n} + \sum_{e(i)=1}p_i,f^q(2,0)R_{\overp,n} + \sum_{e(i)=2}p_i,\dotsc,f^q(2^q,0)R_{\overp,n} + \sum_{e(i)=2^q}p_i\>,
\end{gather*}
with $R_{\overp,n} = 1 - \sum_{i=1}^n p_i$ and $f^q(s,0) = \int_{[0,1]}x^{\gamma_q(s)}(1-x)^{q-\gamma_q(s)}\,d\tau_0(x)$. It is easy to verify that if $p_i = 0$ for $i>n$,
then $\vpt_{n,L_q} = \vpt_{L_q}$. As usual, we will omit $L_q$ whenever it is understood from the
context.

In the proof we will replicate the methods used for the analogous theorem for probability functions
satisfying Predicate Exchangeability, see e.g. Lemma 13 and Theorem 14 in \cite{KliessParis1}.

\begin{proof}
	Suppose that $w = (1+\lambda) w_1 - \lambda w_2$ for some $\lambda \geq 0$ and
	$w_1, w_2$ probability functions satisfying ULi with SPx. Then it is straightforward
	to check that $w$ is a probability function satisfying SPx, but not necessarily
	ULi.
	
	So suppose that $w$ is a probability function satisfying SPx.
	Let $\vec{x}\in\DD{q}$ and consider the function $w_{\vec{x}}$. We can close this function
	under SPx by letting
	\begin{gather}
		z_{\vec{x}} = \frac{1}{|\Sigma|}\sum_{\sigma\in\Sigma} w_{\sigma(\vec{x})},\label{eqz_x}
	\end{gather}
	where $\Sigma$ is the set of all permutations of $\{1,\dotsc,2^q\}$ that preserve $P$-spectra.
	
	It is then easy to check that we obtain a de Finetti representation of $w$ of the form
	\begin{gather}
		w = \int_{\DD{q}} z_{\vec{x}}\,d\mu(\vec{x}).\label{eqRepSpx}
	\end{gather}
	So suppose that $\mu$ is a measure putting all weight on a singleton $\{\vec{x}\}$, i.e.
	$w = z_{\vec{x}}$ for some $\vec{x}\in\DD{q}$.
	
	Pick $\overp$ such that $p_0 = 0$, $p_i = x_i$ for $i=1,\dotsc,2^q$ and $p_i = 0$ for $i > 2^q$.
	Then for arbitrary $\overtau$, $\vpt_{2^q}$ will have $z_{\vec{x}}$ occurring in its representation \eqref{repvptn}, as clearly $e\in Z^q_n$ such that $e(i) = i$ will result
	in $w_{\vec{x}} = w_{e(\overp)}$ occurring, and for each permutation $\sigma$ such that $w_{\sigma(\vec{x})}$
	occurs in $z_{\vec{x}}$ the function $e' = e\circ\sigma\in Z^q_n$ will ensure that $w_{\sigma(\vec{x})}$ occurs in $\vpt_{2^q}$ as well.
	
	Consider the factor $\prod_{r=1}^{2^q}\tau_r^{\gamma_q(e(r))}(1-\tau_r)^{q-\gamma_q(e(r))}$
	accompanying $w_{e(\overp)}$. If $\sigma$ is a permutation of $\{1,\dotsc,2^q\}$ that
	preserves $P$-spectra, then for $e' = e\circ\sigma$ we can easily check that
	\begin{gather}\label{eqAccFactors}
		\prod_{r=1}^{2^q}\tau_r^{\gamma_q(e(r))}(1-\tau_r)^{q-\gamma_q(e(r))} = \prod_{r=1}^{2^q}\tau_r^{\gamma_q(e'(r))}(1-\tau_r)^{q-\gamma_q(e'(r))},
	\end{gather}
	as the atoms $\al_{e(r)}$, $\al_{e'(r)}$ have the same number of negations.
	
	So we have each $w_{\sigma(\vec{x})}$ occurring in $z_{\vec{x}}$ occurring in $\vpt_{2^q}$
	with the same factor. Similarly, for each $\vec{x}\in\DD{q}$ such that $w_{\vec{x}}$ occurs
	in $\vpt_{2^q}$, all the other $w_{\vec{y}}$ occurring in $z_{\vec{x}}$ also occur in $\vpt_{2^q}$, so we can write
	\begin{gather}\label{rep_vptn_z}
		\vpt_{2^q} = \sum_{e\in E} a_{e,\overtau} \cdot S_e\cdot z_{e(\overp)},
	\end{gather}
	where $S_e$ is a normalizing factor, $a_{e,\overtau}$ is the factor $\prod_{r=1}^{2^q}\tau_r^{\gamma_q(e(r))}(1-\tau_r)^{q-\gamma_q(e(r))}$ and $E$ is the collection of representatives of equivalence classes of $Z^q_n$ under the equivalence relation
	$\sim_{Z^q_n}$ defined by
	\begin{gather*}
		e \sim_{Z^q_n} e' \Leftrightarrow w_{e(\overp)}, w_{e'(\overp)}\text{ occur in the same }z_{\vec{x}}.
	\end{gather*}
	
	Note that while
	$\vpt_{2^q}$ depends on both $\overp$ and $\overtau$, the functions $z_{e(\overp)}$ only
	depend on $\overp$ while $a_{e,\overtau}$ only depends on $\overtau$.
	
	Fixing an enumeration of $E$, fix some $\overtau_e$ for each $e\in E$ and let $A$ be the
	$E\times E$ matrix with entry $a_{\<e,f\>} = a_{f,\overtau_e}$. We then obtain the equation
	\begin{gather}\label{eqMatrix1}
		\begin{pmatrix}
			\vdots\\
			v^{\overp,\overtau_e}_{2^q}\\
			\vdots
		\end{pmatrix}
		= A\cdot
		\begin{pmatrix}
			\vdots\\
			S_e\cdot z_{e(\overp)}\\
			\vdots
		\end{pmatrix}.
	\end{gather}
	We will show that the entries of $A$, i.e. the $\overtau_e$, can be picked such that $A$ is regular.
	
	Letting $t = |E|$, define the matrix $A_{\<i_1,\dotsc,i_j\>}$ for $1\leq i_1<i_2<\dotsb<i_j\leq t$ the $j\times j$ sub-matrix of $A$ obtained by taking the $i_1, \dotsc,i_j$'th rows and columns
	of $A$. By induction on $j$, we will show that the $\tau_{s,i_t}$ can be picked such that $A_{\<i_1,\dotsc,i_j\>}$ is regular. It will suffice to obtain for each row products
	of the form $\prod_{r=1}^{2^q}x_r^sy_r^{q-s}$, not necessarily $x+y=1$, nor even $x,y<1$,
	as we can add a factor $(x_r+y_r)^q$ for each $r\in\{1,\dotsc,2^q\}$ to the entry in $A$ to 
	obtain the required form.
	
	For $j = 1$, this is trivial: just pick each $\tau_s$ to be neither $0$ nor $1$. So suppose
	$j = k + 1$ for some $k\geq 1$. The polynomial $\prod_{r=1}^{2^q}x_r^{s_r}y_r^{q-s_r}$ takes its maximum value on $\DD{2^{q+1}}$ at $x_r = s_r/2^{q+1}$, $y_r = q-s_r/2^{q+1}$, $r\in\{1,\dotsc,2^q\}$. Considering the previously fixed enumeration of $E$, there exists
	some $e$ such that
	\begin{gather*}
		\prod_{r=1}^{2^q}(s_{r,e}2^{-q-1})^{s_{r,e}}((q-s_{r,e})2^{-q-1})^{q-s_{r,e}} > \prod_{r=1}^{2^q}(s_{r,e}2^{-q-1})^{s_{r,f}}((q-s_{r,e})2^{-q-1})^{q-s_{r,f}}
	\end{gather*}
	for any $f\neq e$. For if not, then for some $f\neq e$ we have
	\begin{align*}
		\prod_{r=1}^{2^q}(s_{r,e}2^{-q-1})^{s_{r,e}}((q-s_{r,e})2^{-q-1})^{q-s_{r,e}} &\leq \prod_{r=1}^{2^q}(s_{r,e}2^{-q-1})^{s_{r,f}}((q-s_{r,e})2^{-q-1})^{q-s_{r,f}} \\
		&< \prod_{r=1}^{2^q}(s_{r,f}2^{-q-1})^{s_{r,f}}((q-s_{r,f})2^{-q-1})^{q-s_{r,f}}
	\end{align*}
	and continuing in this way, we arrive at a contradiction.
	
	By the inductive hypothesis there are choices for the $x_{r,i_m}, y_{r,i_m}$ for $i_m\in\{i_1,\dotsc,i_j\}\setminus\{i_e\}$ such that the
	sub-matrix $A_{\<i_1,\dotsc,i_{e-1},i_{e+1},\dotsc,i_j\>}$ is regular.
	Treating the $x_r, y_r$ as unknowns for the moment we obtain for the determinant of $A_{\<i_1,\dotsc,i_j\>}$ an expression of the form
	\begin{multline}
		\det(A_{\<i_1,\dotsc,i_j\>}) = \\
		\pm \prod_{r=1}^{2^q}x_r^{s_{r,e}}y_r^{q-s_{r,e}}\cdot\det(A_{\<i_1,\dotsc,i_{e-1},i_{e+1},\dotsc,i_j\>}) + \sum_{f\in\{i_1,\dotsc,i_j\}\setminus\{i_e\}}\prod_{r=1}^{2^q}x_r^{s_{r,f}}y_r^{q-s_{r,f}}\cdot(\pm\det(A_f)),\label{eqsubdet}
	\end{multline}
	for some choices of $\pm$, where the $A_f$ are the corresponding sub-matrices of $A_{\<i_1,\dotsc,i_j\>}$. Picking $x_r = (s_{r,e}2^{-q-1})^g$, $y_r = ((q-s_{r,e})2^{-q-1})^g$
	for large enough $g>0$ the dominant term of \eqref{eqsubdet} becomes
	\begin{gather*}
		\prod_{r=1}^{2^q}x_r^{s_{r,e}}y_r^{q-s_{r,e}}\cdot\det(A_{\<i_1,\dotsc,i_{e-1},i_{e+1},\dotsc,i_j\>}),
	\end{gather*}
	and as for this choice of $x_r, y_r$ we clearly have $\prod_{r=1}^{2^q}x_r^{s_{r,e}}y_r^{q-s_{r,e}}>0$ and by the inductive hypothesis $\det(A_{\<i_1,\dotsc,i_{e-1},i_{e+1},\dotsc,i_j\>}) \neq 0$, we
	get that $\det(A_{\<i_1,\dotsc,i_j\>}\neq 0$. Now introducing factors $\prod_{r=1}^{2^q}(x_r+y_r)^q$ as necessary, we obtain that $A$ is a regular matrix with
	the entries having the required form.
	
	As $A$ is (can be picked to be) regular, we obtain from \eqref{eqMatrix1}
	\begin{gather}\label{eqMatrix2}
		A^{-1}\cdot\begin{pmatrix}
			\vdots\\
			v^{\overp,\overtau_e}_{2^q}\\
			\vdots
		\end{pmatrix}
		= 
		\begin{pmatrix}
			\vdots\\
			S_e\cdot z_{e(\overp)}\\
			\vdots
		\end{pmatrix}
	\end{gather}
	and we can now obtain $z_{\vec{x}} = z_{e(\overp)}$ as a difference of functions $v^{\overp,\overtau_e}_{2^q}$:
	we have
	\begin{align}
		z_{\vec{x}} &= \frac{1}{S_e}\cdot\sum_{e\in E}b_e\cdot v^{\overp,\overtau_e}_{2^q}\nonumber
		\intertext{and collecting the $v^{\overp,\overtau_e}_{2^q}$ with positive coefficients in one term, we obtain}
		z_{\vec{x}} &= \gamma\cdot w_1 - \lambda\cdot w_2,\label{eqzDiff1}		
	\end{align}
	where $w_1$ is a probability function consisting of the $v^{\overp,\overtau_e}_{2^q}$ with positive coefficients and $w_2$ a probability function consisting of the $v^{\overp,\overtau_e}_{2^q}$ with negative coefficients.
	
	Since $z_{\vec{x}}(\top) = w_1(\top) = w_2(\top) = 1$, we must have
	$\gamma - \lambda = 1$, and thus we obtain $\gamma = 1 + \lambda$. As furthermore
	$w_1$, $w_2$ are convex combinations of probability functions satisfying ULi with SPx, we
	have the desired representation for $z_{\vec{x}}$. Note that $\lambda$ depends only on
	$2^q$ and the entries of $A^{-1}$, which in turn only depend on the $\overtau_e$. As $z_{\vec{x}}$ only determines $\overp$, we obtain a uniform $\lambda$ for all $z_{\vec{x}}$.
	
	Now returning to the general case of $w$ of the form \eqref{eqRepSpx}, we obtain
	\begin{align*}
		w &= \int_{\vec{x}} z_{\vec{x}}\,d\mu(\vec{x})\\
		&= \int_{\vec{x}} (1+\lambda)w_{1,\vec{x}} - \lambda w_{2,\vec{x}}\,d\mu(\vec{x})\\
		&= (1+\lambda)\int_{\vec{x}} w_{1,\vec{x}}\,d\mu(\vec{x}) - \lambda\int_{\vec{x}} w_{2,\vec{x}}\,d\mu(\vec{x})\\
		&= (1+\lambda)w_1 - \lambda w_2
	\end{align*}
	for some probability functions $w_1, w_2$ satisfying ULi with SPx, as they are convex combinations of functions satisfying these principles.
\end{proof}

\section{Conclusion}

In this paper we have introduced the principle of Strong Predicate Exchangeability, for which
we have provided two de Finetti -- style representation theorems. While the principle arose
rather by chance, and is more indirectly justified as a rational principle for rational agents,
the representation theorems presented in this paper fit nicely into the story of Pure Inductive Logic so far.

Comparing the representation theorems provided in this paper with similar results for Unary
Language Invariance it appears there is a recurring theme: Looking at presently known results the building blocks
for representing functions satisfying ULi all share the property of Weak Irrelevance.
At the same time, there exist representation theorems for general probability functions satsifying $\Pcal$ showing these
to be \emph{differences} of functions that satisfy ULi with $\Pcal$.

As these principles have so far been based on a symmetry based on the language involved, one might be inclined to expect the same behaviour for other principles based on symmetry, at least where it
concerns purely unary languages. In terms of polyadic languages, there exists a similar result for
Spectrum Exchangeability (generalizing Ax). As the building blocks for the principle SPx discussed in this paper are quite similar to the (unary) $u^{\overp}$ functions, whose polyadic versions play
a role in the representation theorem for Spectrum Exchangeability, one might expect a similar result
for a polyadic version of Px, SPx.

\end{document}